\documentclass{amsart}
\usepackage{amssymb,amsmath,latexsym}
\usepackage{amsthm}
\usepackage{fontenc}
\usepackage{amssymb}
\numberwithin{equation}{section}
\DeclareMathOperator{\spt}{spt}

\newtheorem{theorem}{Theorem}[section]

\newtheorem{lemma}[theorem]{Lemma}

\newcommand{\genlegendre}[4]{%
  \genfrac{(}{)}{}{#1}{#3}{#4}%
  \if\relax\detokenize{#2}\relax\else_{\!#2}\fi
}
\newcommand{\legendre}[3][]{\genlegendre{}{#1}{#2}{#3}}

\begin{document}
\author{Alexander E. Patkowski}
\title{On an spt function for the $4$-th symmetrized crank function}

\maketitle
\begin{abstract} In this paper we find the smallest part function related to the $4$-th symmetrized crank function, corresponding to the one obtained in Patkowski [11] for the $4$-th symmetrized rank function. This provides us with a direct relationship with Garvan's second order smallest part function. We obtain some congruences for these $spt$ functions, as well as asymptotics and inequalities. A finite $q$-series which generates an $spt$ function related to the second crank moment is also stated. This identity has an important relationship to the one obtained by Patkowski [12] related to the second rank moment.\end{abstract}



\section{Introduction and Main Results}

\par Using conventional notation [9], we put $(a)_n=(a;q)_{n}:=\prod_{0\le k\le n-1}(1-aq^{k}),$ so that $(a;q)_{\infty}:=\lim_{n\rightarrow\infty}(a;q)_{n}.$ The now well-known function $\spt(n)$ of Andrews [3], is the number of appearances of the smallest parts in the number of partitions of $n.$ The motivation for its many follow-up studies is due to its relation to the second crank and rank moments, as well as the ordinary partition function $p(n),$ given by $\spt(n)=np(n)-\frac{1}{2}N_2(n).$ Recall that if $N(m,n)$ is the number of partitions of $n$ with rank $m,$ then $N_k(n)=\sum_{m\in\mathbb{Z}}m^kN(m,n).$ One of the remarkable properties this function possesses due to this relationship includes similar congruences to the partition function $p(n).$ Garvan's paper [8] generalized this relation by using the machinery of the Bailey chain (an extension of the Bailey lemma [5]), obtaining many remarkable congruences and inequalities. In [11] and [12] we derived multi-sum and finite $q$-series which also generate certain smallest part functions related to Garvan's higher order functions. In particular, we showed [11, Theorem 1.1] 
$$\sum_{n_1\ge1}\frac{q^{n_1}}{(1-q^{n_1})^2(q^{n_1+1})_{\infty}}\sum_{n_2\ge1}\frac{q^{n_2}}{(1-q^{n_2})^2(q^{n_2+1})_{n_1}}$$
\begin{equation}=\frac{1}{(q)_{\infty}}\left(\sum_{n\ge1}\frac{nq^n}{1-q^n}\right)^2+\frac{1}{(q)_{\infty}}\sum_{\substack{n\in\mathbb{Z}\\ n\neq0}}\frac{(-1)^{n}q^{n(3n+5)/2}}{(1-q^n)^4}.\end{equation}
Using results from [4], we showed this result implies the partition identity [11, Theorem 1.2]
\begin{equation}SPT^{+}(n)=\frac{5}{72}M_4(n)-\frac{1}{6}n^2p(n)+\frac{1}{36}np(n)-\eta_4(n),\end{equation} where $\eta_{k}(n)$ is the $k$-th symmetrized rank [8, pg.242]
$$\eta_k(n)=\sum_{j=-n}^{n}\binom{j+\lfloor{\frac{k-1}{2}}\rfloor}{k}N(j,n).$$
Here $SPT^{+}(n)$ counts the number of partition pairs $(\lambda,\mu)$ with weight $w(\lambda,\mu)$
where $\lambda$ is a partition, $\mu$ is a partition with the difference between its largest parts and the smallest
parts bounded by the smallest part of $\lambda,$ and $w(\lambda,\mu)$ is the product of the number of smallest parts
of $\lambda$ and the number of smallest parts of $\mu.$ 
\par The purpose of this paper is to provide the corresponding smallest part function associated with the $4$-th symmetrized crank function, and its consequences related to Garvan's 2nd order function [8, pg.243] $$\spt_2(n)=\mu_4(n)-\eta_4(n),$$ where $\mu_{k}(n)$ is the $k$-th symmetrized crank [8, pg.242]
$$\mu_k(n)=\sum_{j=-n}^{n}\binom{j+\lfloor{\frac{k-1}{2}}\rfloor}{k}M(j,n).$$ See [1] for important material on the crank on the partition.

\begin{theorem}\label{thm:ex1} We have,
$$\sum_{n_1\ge1}\frac{q^{n_1}}{(1-q^{n_1})^2(q^{n_1+1})_{\infty}}\sum_{n_2\ge1}\frac{q^{n_1n_2+n_2}}{(1-q^{n_2})^2(q^{n_2+1})_{n_1}}$$
\begin{equation}=\frac{1}{(q)_{\infty}}\left(\sum_{n\ge1}\frac{nq^n}{1-q^n}\right)^2+\frac{1}{(q)_{\infty}}\sum_{\substack{n\in\mathbb{Z}\\ n\neq0}}\frac{(-1)^{n}q^{n(n+5)/2}}{(1-q^n)^4}.\end{equation}
\end{theorem}

Let the sequence $SPT^{-}(n)$ be generated by the $q$-series
$$\sum_{n\ge1}SPT^{-}(n)q^n=\sum_{n_1\ge1}\frac{q^{n_1}}{(1-q^{n_1})^2(q^{n_1+1})_{\infty}}\sum_{n_2\ge1}\frac{q^{n_1n_2+n_2}}{(1-q^{n_2})^2(q^{n_2+1})_{n_1}}.$$
Let $s(\lambda)$ denote the the smallest part of the partition $\lambda.$ $SPT^{-}(n)$ counts the number of partition pairs $(\lambda,\omega)$ with weight $w(\lambda,\omega)$ where $\lambda$ is a partition, $\omega$ is a partition with the difference between its largest parts and the smallest parts bounded by the smallest part of $\lambda,$ The smallest part of $\omega$ appears at least $s(\lambda)$ times, and $w(\lambda,\omega)$ is the product of the number of smallest parts
of $\lambda,$ and the number of smallest parts of $\omega$ that appear more than $s(\lambda)$ times. 

The coefficient of $q^n$ in Theorem 1.1 then implies the following theorem.

\begin{theorem} \label{thm:ex2}For each natural number $n,$ we have,
\begin{equation}SPT^{-}(n)=\frac{5}{72}M_4(n)-\frac{1}{6}n^2p(n)+\frac{1}{36}np(n)-\mu_4(n).\end{equation}
\end{theorem}

Garvan [8] established that $$\spt_k(n)=\mu_{2k}(n)-\eta_{2k}(n)>0.$$ Hence putting $k=1$ and subtracting Theorem 1.2 from [11, Theorem 1.2] gives us for all natural numbers $n,$
\begin{equation} \spt_2(n)=SPT^{+}(n)-SPT^{-}(n), \end{equation}
and subsequently
\begin{equation}SPT^{+}(n)>SPT^{-}(n). \end{equation}
We will re-visit (1.6) in our last section when we look into some results on the finite $q$-series related to (1.1) and (1.3) and [12].
A natural consequence of Theorem 1.2 and known congruences for $p(n),$ $M_4(n),$ $\mu_4(n),$ and $\eta_4(n)$ is the following result.

\begin{theorem} \label{thm:ex3} For natural numbers n,
\begin{equation}SPT^{-}(7n)\equiv0\pmod{7},\end{equation}
\begin{equation}SPT^{-}(11n)\equiv0\pmod{11},\end{equation}
\begin{equation}SPT^{-}(7n+5)\equiv0\pmod{7},\end{equation}
\begin{equation}SPT^{+}(7n+5)\equiv0\pmod{7}.\end{equation}
\end{theorem}
\begin{proof} By [8, Theorem 6.1, eq.(6.2),eq.(6.3)], and [11, Theorem 1.3] we obtain (1.7) and (1.8) from the relationship (1.5). Note from [8, pg.251], 

\begin{equation} \mu_4(n)=\frac{1}{4!}\sum_{m=-n}^{n}(m^2-1)m^2M(m,n)=\frac{1}{4!}(M_4(n)-M_2(n)).\end{equation}

From [8, pg.256],
\begin{equation}M_4(7n+5)\equiv M_2(7n+5)\equiv0\pmod{7}.\end{equation}
Hence from (1.11), (1.12), $M_2(n)=2np(n),$ and Theorem 1.2 we obtain (1.9). Likewise, using (1.9) with [8, Theorem 6.1, eq.(6.2)] and (1.5) gives (1.10). 

\end{proof}
In a recent paper by C. Williams [13], many elegant congruences for Garvan's higher order smallest part function were established. Let $\legendre{-n}{l}=1$ for a prime $l\ge5,$ then Williams proved that [13, Theorem 1.1] for $m\ge1,$
\begin{equation} \spt_2\left(\frac{l^{2m}n+1}{24}\right)\equiv0\pmod{l^{m}}.\end{equation} This result is reminiscent of a known result on the partition function [13, eq.(1.1)]
\begin{equation} p\left(\frac{l^{2m}n+1}{24}\right)\equiv0\pmod{l^{2m}}.\end{equation} (Here we hold the convention that $p(n)$ and $\spt(n)$ are zero for $n\notin\mathbb{N}.$) Utilizing (1.13) and related results from Williams [13] we arrive at the following general result.
\begin{theorem} Let $\legendre{-n}{l}=1$ for a prime $l\ge5,$ then
\begin{equation} SPT^{-}\left(\frac{l^{2m}n+1}{24}\right)\equiv0\pmod{l^{m}},\end{equation}
\begin{equation}SPT^{+}\left(\frac{l^{2m}n+1}{24}\right)\equiv0\pmod{l^{m}}.\end{equation}
\end{theorem}
\begin{proof} From [13, pg.337] we have
\begin{equation} 4^2\times5\times M_4\left(\frac{l^{2m}n+1}{24}\right)\equiv p\left(\frac{l^{2m}n+1}{24}\right)\pmod{l^{2m}},\end{equation}
and
\begin{equation} 2^7\times3^2\times5\times \mu_4\left(\frac{l^{2m}n+1}{24}\right)\equiv -17p\left(\frac{l^{2m}n+1}{24}\right)\pmod{l^{2m}}.\end{equation}
If $la\equiv b\pmod{l^{2m}},$ and $b\equiv0\pmod{l^{2m}},$ then $a\equiv0\pmod{l^{m}}.$ To see this note that $la\equiv b\pmod{l^{2m}},$ implies $la=b+n_1 l^{2m},$ for an integer $n_1.$ Further $b\equiv0\pmod{l^{2m}},$ implies $la=b+n_1 l^{2m}=n_2 l^{2m}+n_1 l^{2m},$ or equivalently $a=(n_2+n_1) l^{2m-1}=(n_2+n_1)l^{m-1} l^{m}\equiv0\pmod{l^m}.$

Hence, (1.17) and (1.18) with (1.14) imply for a prime $l\ge5,$
\begin{equation}  M_4\left(\frac{l^{2m}n+1}{24}\right)\equiv 0\pmod{l^{m}},\end{equation}
and
\begin{equation}  \mu_4\left(\frac{l^{2m}n+1}{24}\right)\equiv 0\pmod{l^{m}}.\end{equation}
Now applying (1.14), (1.19) and (1.20) with Theorem 1.2 gives (1.15). Williams's result (1.13) with (1.15) gives (1.16).
\end{proof}

In Bringmann and Mahlburg [7, Theorem 1.2, (i)] it was proved that
\begin{equation}\frac{1}{2}M_{2k}(n)\sim \frac{1}{2}N_{2k}(n)\sim \gamma_{2k} n^{k-1}e^{\pi\sqrt{\frac{2n}{3}}}, \end{equation}
as $n\rightarrow\infty,$ for $k\ge1,$
and $$\gamma_{2k}=(2k)!\zeta(2k)(1-2^{1-2k})\frac{6^{k}}{4\sqrt{3}\pi^{2k}}. $$ Here as usual $\zeta(k)$ is the Riemann zeta function. We also used the common notation $a_n\sim b_n$ to mean the ratio $a_n/b_n$ tends to $1$ as $n\rightarrow\infty.$  It should be noted the asymptotic formula (1.21) is very similar to the known one for the partition function [7, eq.(1.15)]

\begin{equation}p(n)\sim \frac{1}{4\sqrt{3}n}e^{\pi\sqrt{\frac{2n}{3}}}, \end{equation}
as $n\rightarrow\infty.$ In fact, (1.22) can be obtained from (1.21) when $k=1.$ The formulas for $SPT^{\pm}(n)$ in conjunction with (1.21) suggests the growth as $n\rightarrow\infty$ would be the same due to the fourth moments. This is presented in the following result.
\begin{theorem} We have,
\begin{equation}SPT^{+}(n)\sim SPT^{-}(n)\sim \frac{\sqrt{3}}{40}ne^{\pi\sqrt{\frac{2n}{3}}}, \end{equation}
as $n\rightarrow\infty.$
\end{theorem}
\begin{proof} First note that (1.21) tells us the growth of $M_4(n)$ and $N_4(n)$ dominates the growth of $M_2(n)$ and $N_2(n).$ Hence by (1.21), (1.22), Theorem 1.2 tells us 
\begin{equation}\begin{aligned}SPT^{-}(n)&\sim\frac{5}{72}M_4(n)-\frac{1}{6}n^2p(n)-\frac{1}{4!}M_4(n)\\
&\sim \left(\frac{5}{72}\times 2\gamma_4-\frac{1}{24\sqrt{3}}-\frac{2}{24}\gamma_4 \right)ne^{\pi\sqrt{\frac{2n}{3}}}.\end{aligned}\end{equation}
Similarly using $\eta_4(n)=(N_4(n)-N_2(n))/4!$ gives us the same asymptotic for $SPT^{+}(n).$ After some simplification of the constants gives the result (1.23).
\end{proof}

\section{Proofs of Theorems 1.1 and 1.2}

To prove our $q$-series identities analytically, we will use a special case of differentiating each variable of the $2$-fold Bailey lemma given in [11] relative to $a_1=a_2=1.$

\begin{lemma}
({\bf [2, Theorem 1],[11]}) \it Define a pair of sequences $(\alpha_{n_1, n_2},\beta_{n_1, n_2})$ to be a $2$-fold Bailey pair with respect to $a_j,$ $j=1,2,$ if 
		\begin{equation}\beta_{n_1, n_2}=\sum_{r_1\ge0}^{n_1}\sum_{r_2\ge0}^{n_2} \frac{\alpha_{r_1,r_2}}{(a_1q;q)_{n_1+r_1} (q;q)_{n_1-r_1}(a_2q;q)_{n_2+r_2} (q;q)_{n_2-r_2}}.\end{equation}
For $a_1=a_2=1,$ we have
\begin{equation}\sum_{n_1\ge1}\sum_{n_2\ge1}(q)_{n_1-1}^2(q)_{n_2-1}^2\beta_{n_1, n_2}q^{n_1+n_2}=\left(\sum_{n\ge1}\frac{nq^n}{1-q^n}\right)^2+\sum_{n_1\ge1}\sum_{n_2\ge1}\frac{\alpha_{n_1, n_2}q^{n_1+n_2}}{(1-q^{n_1})^2(1-q^{n_2})^2}.\end{equation} 
\end{lemma}
\rm

\begin{proof}[Proof of Theorem \ref{thm:ex1}] Joshi and Vyas [10] established that $(\alpha_{n_1, n_2},\beta_{n_1, n_2})$ is a $2$-fold Bailey pair with respect to $a_j=1,$ $j=1,2,$ where $\alpha_{0, 0}=1,$ where

\begin{equation}\alpha_{n_1, n_2}=\begin{cases} (-1)^nq^{n(n-1)/2}(1+q^n),& \text {if } n_1=n_2=n,\\ 0, & \text{otherwise,} \end{cases}\end{equation}
and 

\begin{equation}\beta_{n_1, n_2}=\frac{q^{n_1n_2}}{(q)_{n_1}(q)_{n_2}(q)_{n_1+n_2}}.\end{equation}
Inserting this $2$-fold Bailey pair (2.3)--(2.4) into Lemma 2.1 and then multiplying through by $(q)_{\infty}^{-1},$ gives
$$\frac{1}{(q)_{\infty}}\sum_{n_1,n_2\ge1}\frac{(q)_{n_1-1}^2(q)_{n_2-1}^2q^{n_1+n_2+n_1n_2}}{(q)_{n_1}(q)_{n_2}(q)_{n_1+n_2}}=\sum_{n_1\ge1}\frac{q^{n_1}}{(1-q^{n_1})^2(q^{n_1+1})_{\infty}}\sum_{n_2\ge1}\frac{q^{n_1n_2+n_2}}{(1-q^{n_2})^2(q^{n_2+1})_{n_1}}.$$
After a little rearrangement we obtain the stated identity.
\end{proof}

\begin{proof}[Proof of Theorem \ref{thm:ex2}]The first sum in Theorem 1.1 is the generating function for $SPT^{-}(n).$ The first product times sum on the right hand side follows from the identity established in [11]
  \begin{equation}P\Phi_1^2=\frac{5}{72}C_4-\frac{1}{6}\sum_{n\ge1}n^2p(n)q^n+\frac{1}{36}P\Phi_1,\end{equation}
  where
  $$\Phi_i:=\sum_{n\ge1}\frac{n^iq^n}{1-q^n}.$$ The coefficient of $q^n$ in $C_4$ is $M_4(n).$
The last $q$-series of Theorem 1.1 is the $k=2$ instance of the definition [8]
\begin{equation}\sum_{n\ge0}\mu_{2k}(n)q^n=\frac{1}{(q)_{\infty}}\sum_{\substack{n\in\mathbb{Z}\\ n\neq0}}\frac{(-1)^{n+1}q^{n(n+1)/2+kn}}{(1-q^n)^{2k}}.\end{equation}
Equating coefficients of $q^n$ of Theorem 1.1 gives the partition identity.
\end{proof}

\section{The finite $\spt$ functions }
Here we obtain the finite $\spt$ function for the 2nd crank moment, corresponding to the finite $\spt$ function for the 2nd rank moment obtained in [12]. Specifically, the generating function from [12, Corollary 4] is $spt^{*}_M(n),$ the number of smallest parts of the number of partitions of $n$ where parts greater than the smallest plus $M$ do not occur. We will utilize the $1$-fold Bailey lemma created in [5]. A pair of sequences $(\alpha_n,\beta_n)$ is referred to as a Bailey pair [4] with respect to $a$ if
\begin{equation}\beta_n=\sum_{0\le j\le n}\frac{\alpha_j}{(q;q)_{n-j}(aq;q)_{n+j}}.\end{equation}
From [12, eq.(2.5)] we know that if $(\alpha_n,\beta_n)$ is a Bailey pair relative to $a=1,$ then 
\begin{equation}\sum_{n\ge1}(q)_{n-1}^2\beta_nq^n=\alpha_0\sum_{n\ge1}\frac{nq^n}{1-q^n}+\sum_{n\ge1}\frac{\alpha_nq^n}{(1-q^n)^2}. \end{equation}

\begin{theorem} We have, for each natural number $M,$
\begin{align}&\sum_{n_1\ge1}\frac{q^{nM+n}}{(1-q^{n})^2(1-q^{n+1})\cdots(1-q^{n+M})}\\ &\quad =\frac{1}{(q)_{M}}\sum_{n\ge1}\frac{nq^n}{1-q^n}+(q)_{M}\sum_{n\ge1}\frac{(-1)^nq^{n(n+1)/2}(1+q^n)}{(q)_{M-n}(q)_{M+n}(1-q^n)^2}.\notag \end{align}
\end{theorem}
\begin{proof} From [6, pg.15, eq.(6.1)], for $|t|<1,$
\begin{equation}\sum_{n\ge0}\frac{(a)_n(b)_nt^n}{(c)_n(q)_n}=\frac{(abt/c)_{\infty}}{(t)_{\infty}}\sum_{n\ge0}\frac{(c/a)_n(c/b)_n\left(\frac{abt}{c}\right)^n}{(c)_n(q)_n}. \end{equation}
If we let $a, b\rightarrow 0,$ $c=q^{M+1},$ in (3.4) and simplify we obtain
\begin{equation}\sum_{n\ge0}\frac{(q)_M t^n}{(q)_{M+n}(q)_n}=\frac{(q)_{M}}{(t)_{\infty}}\sum_{n\ge0}\frac{q^{n^2+Mn}t^n}{(q)_{n+M}(q)_n}. \end{equation}
The coefficient of $t^n$ of (3.5) is 
\begin{equation}\frac{1}{(q)_{M+n}(q)_n}=\sum_{k\ge0}\frac{q^{k^2+Mk}}{(q)_{n-k}(q)_{k+M}(q)_k}. \end{equation}
Hence applying (3.6) with [12, Lemma 3] and the uniqueness of Bailey pairs with Bailey's lemma [12,Theorem 3.2, $\rho_1\rightarrow\infty,\rho_2\rightarrow\infty$]
$$\alpha_n=q^{n^2}\alpha'_n,$$
$$\beta_n=\sum_{k\ge0}^{n}\frac{q^{k^2}}{(q)_{n-k}}\beta'_k,$$
we see that $(\alpha_n,\beta_n)$ is a Bailey pair relative to $a=1$ where,
\begin{equation}\alpha_n=\frac{(q)_{M}^2(-1)^nq^{n(n-1)/2}(1+q^n)}{(q)_{M-n}(q)_{M+n}}, \end{equation}
for $1\le n\le M,$ $\alpha_n=0$ for $n>M,$ $\alpha_0=1,$ and
\begin{equation}\beta_n=\frac{q^{nM}(q)_M}{(q)_{n+M}(q)_n}. \end{equation}
Applying the Bailey pair (3.7)--(3.8) to (3.2) gives the theorem.
\end{proof}

The generating function on the left side of (3.3) may be observed to be related to the $\omega$ component of the partition pair in Theorem 1.2. It generates $spt^{-}_M(n),$ the number of smallest parts that appear more than $M$ times, of the number of partitions of $n$ where parts greater than the smallest plus $M$ do not occur, and the smallest part appears at least $M$ times. The set of partitions in which the smallest part appears at least $M$ times, largest part bounded by the smallest plus $M$ is a subset of partitions in which the largest part is bounded by the smallest plus $M.$ Furthermore, the number of smallest parts that appear more than $M$ times is less than the total number of smallest parts that appear in any partition. Hence, combinatorially we have shown the result.
\begin{theorem} For natuaral numbers $n,M>0,$
\begin{equation}spt^{*}_M(n)>spt^{-}_M(n).\end{equation}
\end{theorem}
It may be seen from the component $q^{nM}$ in the numerator of the generating function for $spt^{-}_M(n)$ that $\lim_{M\rightarrow\infty}spt^{-}_M(n)=0.$ Combinatorially, there are clearly no such partitions for which the smallest part appears an infinite number of times. Hence a direct corollary tells us that $\lim_{M\rightarrow\infty}spt^{+}_M(n)=\spt(n)>0.$ Another interesting consequence of Theorem 3.2 is equation (1.6), which again can follow from our combinatorial argument.

1390 Bumps River Rd. \\*
Centerville, MA
02632 \\*
USA \\*
ul. A. E. Ody\'{n}ca 47 \\*
02-606 Warsaw\\*
Poland\\*
E-mail: alexpatk@hotmail.com, alexepatkowski@gmail.com

\begin{thebibliography}{9}
 

\bibitem{ConcreteMath}
G. E. Andrews and F. Garvan, \emph{Dyson's crank of a partition,} Bull. Amer. Math. Soc. (N.S.) 18 (1988), no. 2, 167--171.
\bibitem{ConcreteMath}
G. E. Andrews, \emph{Umbral Calculus, Bailey Chains, and Pentagonal Number Theorems,} J. Comb. Theory, Ser. A 91 (1-2): 464--475 (2000)

\bibitem{ConcreteMath}
G. E. Andrews, \emph{The number of smallest parts in the partitions of n,} J. Reine Angew. Math. 624 (2008), 133-142.
\bibitem{ConcreteMath}
A.O.L. Atkin and F.G. Garvan, \emph{Relations between the ranks and cranks of partitions,} (Rankin memorial issues) Ramanujan J. 7 (2003), no. 1-3, 343--366.
\bibitem{ConcreteMath}
 W. N. Bailey, \emph{Identities of the Rogers--Ramanujan type}, Proc. London Math. Soc. (2), 50 (1949), 1--10.
 
\bibitem{ConcereteMath} B. C. Berndt, Ramanujan's Notebooks, Part III. New York: Springer--Verlag, 1991.
\bibitem{ConcreteMath}
K. Bringmann, K. Mahlburg, \emph{Asymptotic inequalities for positive crank and rank moments,} Transactions of the American Mathematical Society 366 (2014), pages 1073-1094.
\bibitem{ConcreteMath}
F. Garvan, \emph{Higher Order spt--Functions,} Adv. in Math. 228 (2011), 241-265.
\bibitem{ConcreteMath}
G. Gasper,  M. Rahman, \emph{Basic hypergeometric series}, Cambridge Univ. Press, Cambridge, 1990.

 \bibitem{ConcreteMath}
C. M. Joshi, Y. Vyas,  \emph{Bailey Type Transforms and Applications,} $J\tilde{n}\tilde{a}n\tilde{a}bha$, 45 (2015), pp. 53--80.

\bibitem{ConcreteMath}
A. E. Patkowski, \emph{An interesting $q$-series related to the 4th symmetrized rank function,} Discrete Mathematics, Volume 341, Issue 11, 2018, pages 2965--2968.

\bibitem{ConcreteMath}
A. E. Patkowski, \emph{Some Implications of the $2$-fold Bailey lemma,} Methods and Applications of Analysis, Vol. 27, No. 2, pp. 153--160, 2020.
\bibitem{ConcreteMath}
C. Williams, \emph{Hecke relations for eta multipliers and congruences for higher-order smallest parts functions,} Journal of Number Theory, Volume 277, 2025, Pages 325-343.

\end{thebibliography}
\end{document}